\documentclass[11pt]{article}

\usepackage{my_style}

\title{Average size of 2-Selmer groups\\ of elliptic curves over function fields}
\author{Q.P. H\`{\^o}, V.B. L\^{e} H\`{u}ng, B.C. Ng\^{o}}

\begin{document}
	\maketitle

	\begin{abstract}
		Employing a geometric setting inspired by the proof of the Fundamental Lemma, we study some counting problems related to the average size of 2-Selmer groups and hence obtain an estimate for it.
	\end{abstract}
	
	
	\section{Introduction}
	
	By the Mordell-Weil theorem, for every elliptic curve $E$ over a global field $K$, the group $E(K)$ of $K$-rational points of $E$ is a finitely generated abelian group. The rank of $E(K)$, called the Mordell-Weil rank, is a fascinating invariant as revealed by the Birch and Swinnerton-Dyer conjecture. It remains nevertheless very mysterious. For instance, it is not known if the  Mordell-Weil rank of elliptic curves defined over a given number field is bounded. Over function fields, according to Ulmer \cite{ulmer_elliptic_2002}, the Mordel-Weil rank is known to be unbounded.
		
	In the ground breaking papers \cite{bhargava_binary_2010} and \cite{bhargava_ternary_2010}, Bhargava and Shankar were able to prove an upper bound for the average rank of $E(\mathbb{Q})$, when $E$ ranges over the set of elliptic curves defined over $\mathbb{Q}$.
	
	An attractive feature of their work is its rather elementary nature. Bhargava and Shankar bound the average rank by estimating the average size of the 2-Selmer groups $\Sel_2(E)$ of $E$. This computation is then carried out as the solution of a problem in geometry of numbers which involves counting integral points in a certain fundamental domain built out of the action of $\PGL_2$ on the space of binary quartic polynomials.
	
	The aim of this work is to introduce certain moduli spaces, also built out of the action of $\PGL_2$ on binary quartics, which should be viewed as the geometric analog of this problem in geometry of numbers in the case of global fields of rational functions on a curve defined over a finite field. Counting points on these moduli spaces, which is roughly counting torsors for suitable quasi-finite group schemes over the curve, will then help to estimate the average size of 2-Selmer groups, and hence the average rank of elliptic curves. This gives a (weakened) function field analog of the main result of \cite{bhargava_binary_2010}, valid for all functions fields with very mild restrictions.
  
  \begin{thm*}
    Let $K$ be a global function field over a finite field $\mathbb{F}_q$ with $q>32$ and $\car \mathbb{F}_q > 3$. Then the average size of 2-Selmer groups of elliptic curves over $K$ when ordered by height, is bounded above and below by explicit functions $3+F(q)$ and $3-G(q)$. Furthermore $F(q)$, $G(q)$ tend to 0 as $q\to \infty$.     
  \end{thm*}

  More precise statements of our result are given in subsection \ref{subsec:the_results}. We also remark that the results of \cite{de_jong_counting_2002} give upper bounds for the size of 3-Selmer groups of a similar nature for the case $K = k(\mathbb{P}^1) = \mathbb{F}_q(t)$. After the completion of this paper, we learned from J. Ellenberg that Y. Zhao in \cite{zhao_sieve_2013} has also obtained results in the case of cubic polynomials using an argument which is in part similar to ours. It seems that our methods may be applicable to more general coregular representations, for example the ones studied in \cite{jack_thorne_vinbergs_2013}, and we hope to return to this in future work.

\subsection*{Acknowledgement}
This work is partially supported by the NSF grant DMS-1302819 and a Simons investigator's grant of B.C. Ng\^o. The work started during a summer seminar on the work \cite{bhargava_binary_2010}, organized by B.C. Ng\^{o} at the Vietnam Institute for Advanced Studies in Mathematics (VIASM). V.B. L\^{e} H\`{u}ng and Q. H\`{\^o} would like to thank the VIASM for its hospitality. V.B. L\^{e} H\`{u}ng would like to thank the University of Chicago for its support and hospitality during a visit where part of this work was done. We thank the referees for their careful reading of our manuscript.
	
	\paragraph{Notations:} $k = \mathbb{F}_q$ with $\car k \neq 2, 3$, $\lbar{k}$ its algebraic closure, $C$ is a smooth, complete, geometrically connected curve over $k$ such that $C(k) \neq \emptyset$, $K = k(C)$, the field of rational functions on $C$, and $G = \PGL_2$.
	
	\section{Elliptic curves over $K$} \label{sec:Weierstrass_models}
We will need to specify an ordering on the infinite set of isomorphism classes of elliptic curves over $K = k(C)$ in order to make sense of the notion of average. This can be done via the notion of height, which in turn relies on the theory of minimal Weierstrass models of elliptic curves.	
	
	\subsection{Height and minimal Weierstrass model}
	
	We will recall the statements of the necessary results of the theory of Weierstrass model, and refer the readers to the literature for the proofs.

	\begin{defn}
    A family of Weierstrass curves over a scheme $S$ is a flat family of arithmetic genus one curves $\pi:E\to S$ with integral geometric fibers, equipped with a section $e:S\to E$ not passing through the cusps or nodes of any fiber.
  \end{defn}
  
  A family of Weierstrass curves admits a simple presentation, which justifies its name.
  
  \begin{prop}\label{prop:weierstrass_form_in_family}
    Let $(E, e)$ be a family of Weierstrass curves over a scheme $S$. Then, there exists a triple $(\mathcal{L}, a, b)$ with $\mathcal{L}$ a line bundle over $S$, $a\in H^0(C, \mathcal{L}^{\otimes 4})$ and $b\in H^0(C, \mathcal{L}^{\otimes 6})$ such that the pair $(E, e)$ is isomorphic to the closed subscheme of $\mathbb{P}(\mathcal{L}^{\otimes -2} \oplus \mathcal{L}^{\otimes -3} \oplus \mathcal{O}_C)$ defined by the equation
    \[
      yz^2 = x^3 + axz^2 + bz^3,
    \]
    and the section $e: S\to E$ is given by $(0, 1, 0)$.
    
    Moreover, $(\mathcal{L}, a, b)$ is unique up to the following identification: $(\mathcal{L}, a, b) \sim (\mathcal{L}', a', b')$ when $\mathcal{L}\cong \mathcal{L}'$ and $(a, b) = (c^4 a, c^6 b)$ for some $c\in k^\times$. In particular, $(E, e)$ completely determines $\mathcal{L}$, and in fact, $\mathcal{L} = \pi_*(\mathcal{O}_E(e)/\mathcal{O}_E)^{-1}$.
  \end{prop}
  \begin{proof}
    See~\cite[theorem 2.1]{miranda_moduli_1981} and~\cite{suominen_introduction_1970}.
  \end{proof}

	\begin{rmk} \label{rmk:univerrsal_weierstrass} Proposition~\ref{prop:weierstrass_form_in_family} allows us to construct the moduli stack of Weierstrass curves as the stack quotient $[\mathbb A^2/\Gm]$, with $\Gm$ acting on $\mathbb A^2$ by the formula $c\cdot(a,b)=(c^4 a,c^6b)$. The universal family is the closed subscheme of $\mathbb{P} (\mathcal{L}_{uni}^{\otimes -2} \oplus \mathcal{L}_{uni}^{\otimes -3} \oplus \mathcal{O}_C)$ cut out by the equation $yz^2 = x^3 + axz^2 + bz^3$, with $\mathcal L_{uni}$ being the pullback of the universal line bundle on $B\Gm$, and the section $e:S\to E$ given by $(0, 1, 0)$.   
  \end{rmk}

  \begin{thm}\label{thm:minimal_Weierstrass}
    Let $(E_K, e_K)$ be an elliptic curve over $K$. Then, we can extend $(E_K, e_K)$ to a family of Weierstrass curves $(E, e)$ over $C$. Moreover, the extension is unique up to isomorphism if we demand that the line bundle $\mathcal{L}=\pi_*(\mathcal{O}_E(e)/\mathcal{O}_E)^{-1}$  (see proposition~\ref{prop:weierstrass_form_in_family}) is of minimal degree.
  \end{thm}
	\begin{proof}
		See~\cite[section 9.4]{liu_algebraic_2006}.
	\end{proof}

\begin{defn}
The height of an elliptic curve $E_K$ defined over $K$ is defined to be the minimal $\deg\mathcal{L}$ in the theorem above.
\end{defn}

Using proposition~\ref{prop:weierstrass_form_in_family} and remark~\ref{rmk:univerrsal_weierstrass}, theorem \ref{thm:minimal_Weierstrass} can now be reformulated in a slightly different way. Every elliptic $(E_K,e_K)$ over the generic point $\Spec(K)$ of $C$ can be extended as a family of Weierstrass curves $(E,e)$ over $C$, and hence gives rise to a morphism $h_E:C\to [\mathbb A^2/\Gm]$. The extension is unique if $\deg h_E^* \mathcal L_{uni}$ is minimal.

Let $(E,e)$ be a family of Weierstrass curves over $C$. Then the fiber $E_v$ over a point $v\in C$ is singular if and only if $v$ lies in the zero divisor of the discriminant 		
		$$\Delta(a,b) = -(4a^3+27b^2) \in \Gamma(C, \mathcal{L}^{\otimes 12}).$$  
We will sometimes use the notation $\Delta(E_K)$ to denote the discriminant of  the minimal Weierstrass model.
	
\begin{defn} \label{transversal}
A morphism $\alpha:C\to [\mathbb{A}^2/\Gm]$ is said to be transversal to the discriminant locus if the zero divisor of $\alpha^*\Delta = 4a^3 + 27b^2 \in \Gamma(C, \mathcal{L}^{\otimes 12})$ is multiplicity free.	
\end{defn}
		
\subsection{Statements of the main theorems} \label{subsec:the_results}

We recall that for each  elliptic curve $E$ defined over $K$, the 2-Selmer group of $E$ is defined as the kernel of the homomorphism:
\[
  \Sel_2(E)={\rm ker}(H^1(K, E[2]) \to \prod_{v\in |C|} H^1(K_v, E)).
\]
We will now state the main results of the paper. First, we introduce the following notation:
	\[
		\AS(d) = \disfrac{\sum_{h(E_K) \leq d} \frac{|\Sel_2(E_K)|}{|\Aut(E_K)|}}{\sum_{h(E_K) \leq d} \frac{1}{|\Aut(E_K)|}}
		\qquad \text{and} \qquad 
		\AR(d) = \disfrac{\sum_{h(E_K) \leq d} \frac{|\Rank(E_K)|}{|\Aut(E_K)|}}{\sum_{h(E_K) \leq d} \frac{1}{|\Aut(E_K)|}}. \teq \label{eq:AS_AR_formula}
	\]
	Similarly, we denote $\AS(\mathcal{L})$ and $\AR(\mathcal{L})$ to be similar to $\AS(d)$ and $\AR(d)$ except that we restrict ourselves to those elliptic curves whose minimal models are given by a fixed line bundle $\mathcal{L}$ (see theorem~\ref{thm:minimal_Weierstrass}). Note that it makes sense to talk about $\AS$ and $\AR$ since the number of isomorphism classes of elliptic curves over $K$ with bounded height is finite.

In all the results below, we make the assumption that the base field $k$ has more than 32 elements. The source of this restriction will be explained in subsection~\ref{subsec:E[2](C)_nontrivial}.
	
	\begin{thm} \label{thm:main}
		We have the following bounds for $\AS(\mathcal{L})$:
		$$
		 \limsup_{\deg \mathcal{L} \to \infty} \AS(\mathcal{L}) \leq 3 + \frac{T}{(q-1)^2},
		$$
		and
		$$
		  \liminf_{\deg \mathcal{L} \to \infty} \AS(\mathcal{L}) \geq 3\zeta_C(10)^{-1},
		$$
where $T$ is a constant depending only on $C$, and $\zeta_C$ is the zeta function associated to $C$. 
	\end{thm}
	
	From this theorem, we derive the following corollaries.
	\begin{cor} \label{cor:main}
		If we order elliptic curves over $K$ by height, then we have
		$$
			\limsup_{d\to \infty}	\AS (d) \leq 3 + \frac{T}{(q-1)^2},
		$$
		and
		$$
		  \liminf_{d\to \infty} \AS(d) \geq 3\zeta_C(10)^{-1}.
		$$
		In particular,
		$$
			\lim_{q\to \infty} \limsup_{d\to \infty} \AS (d) \leq 3,
		$$
		and
		$$ 
		  \lim_{q\to \infty} \liminf_{d\to \infty} \AS(d) \geq 3.
		$$
	\end{cor}
	\begin{proof}
		This is clear from theorem~\ref{thm:main}, noticing that $\lim_{n\to \infty} \zeta_{C \otimes {{\mathbb F}_{q^n}}}(10) = 1$.
	\end{proof}
	
	\begin{cor} We have the following bounds for the average rank:
		$$
			\limsup_{d\to \infty} \AR(d) \leq \frac{3}{2} + \frac{T}{2(q-1)^2}.
		$$
		In particular,
		$$
			\lim_{q\to \infty} \limsup_{d\to \infty} \AR(d) \leq \frac{3}{2},
		$$
	\end{cor}
	\begin{proof}
		This is a direct consequence of corollary~\ref{cor:main}.
	\end{proof}
	
	If we restrict ourselves to the case where $\Delta(E_K)$ square-free, then we get a better estimate for the average size of the 2-Selmer groups, and hence, also for the average rank. For the sake of brevity, we add the superscript $sf$ to $\AS^{sf}(d)$, $\AR^{sf}(d)$, $\AS^{sf}(\mathcal{L})$ and $\AR^{sf}(\mathcal{L})$ to mean that we restrict the range to the cases where $\Delta(E_K)$ is square-free.
		
	\begin{thm} \label{thm:transversal}
		When we restrict ourselves to the square-free range, then
		$$
			\lim_{\deg\mathcal{L} \to \infty} \AS^{sf}(\mathcal{L}) = 3,
		$$
		and hence
		$$
			\lim_{d\to \infty} \AS^{sf}(d) = 3,
		$$
		and
		$$
			\lim_{d\to \infty} \AR^{sf}(d) \leq \frac{3}{2}.
		$$
	\end{thm}

The rest of the paper will be devoted to the proofs of theorems~\ref{thm:main} and~\ref{thm:transversal}. The main strategy to our counting problem is the introduction of a morphism of stacks $\mathcal{M}_\mathcal{L} \to \mathcal{A}_\mathcal{L}$ parametrized by line bundles $\mathcal{L}$ on $C$, and calculate the limit  of the ratio of masses 
$$ |\mathcal M_{\mathcal L}(k)|/ |\mathcal{A}_\mathcal{L}(k)| $$
as $\deg(\mathcal L)\to\infty $. This geometric situation will be set up in subsection \ref{subsec:geometric_setup} after some necessary preparations.

	\section{Invariant theory of binary quartic forms}
	
	\subsection{Invariants}
	Let $V = \Spec k[c_0,c_1,c_2,c_3,c_4]$ be the space of binary quartic forms with coefficients $c_0,c_1,c_2,c_3,c_4$, i.e. a point $f\in V(k)$ can be written as
\[
		f(x, y) = c_0x^4 + c_1x^3y + c_2x^2y^2 + c_3xy^3 + c_4y^4.
\]
We can view $V$ as a representation of $\GL_2$ by identifying $V$ with $\Sym^4 \std \otimes \det^{-2}$, where $\std$ stands for the standard representation of $\GL_2$. The center of $\GL_2$ acts trivially on $V$, which makes this into a representation of $G=\PGL_2$. From the classical theory of invariants, we know that the GIT quotient $\geoquot{V}{G}$ of $V$ is isomorphic to $S = \Spec k[a, b]$, where
	\begin{align*}
		a &= -\frac{1}{3} (12c_0 c_4 - 3c_1 c_3 + c_2^2),\\
		b &= -\frac{1}{27} (72c_0c_2c_4 + 9c_2c_3c_4 - 27c_0c_3^2 - 27c_4c_1^2 - 2c_2^3),
	\end{align*}
and we denote $\pi: V\to S$ the quotient map. The discriminant
	\[
	   \Delta(f) = -(4a^3+27b^2)
	\]
defines regular functions on $V$ and $S$.
	
We also have a linear action of $\Gm$ on $V$ and $S$ compatible with $\pi$ and with the $G$-action defined as follows
\[ 
		c\cdot f = c^2 f \qquad \text{and}\qquad c\cdot(a,b) = (c^4 a, c^6 b). \teq \label{Gm-action}
\]
These relations induce a natural morphism of quotient stacks $\pi: [V/G\times \Gm] \to [S/\Gm]$. We also have the relation:
\[
  c\cdot \Delta=c^{12} \Delta
\]
which implies that $\Delta$ defines a divisor on $[S/\Gm]$.

The quotient map $\pi$ admits a section $s$ given by
\[
		s(a, b) = y(x^3+axy^2 +by^3), \teq \label{Weierstrass-section}
\]
which we will call the Weierstrass section. In fact, this section can be extended to a map $S\times \Gm \to V\times G\times \Gm$ compatible with all the actions involved
	$$
		s((a, b), c) = \left(y(x^3+axy^2+by^3), \mtrix{1 & 0 \\ 0 & c^2}, c\right).
	$$
This section induces a section on the level of quotient stacks:
$$[S/\Gm] \to [V/ G\times \Gm]$$	
also to be called the Weierstrass section.
	
	\subsection{Stable orbits}
We will now investigate the orbits and stabilizers of the action of $G$ on the space of binary quartic forms. A non-zero binary quartic form $f \in V(\lbar{k})$ can be written in the following form:
	\[
		f(x, y) = \prod_{i=1}^4 \, (a_ix + b_i y), \qquad a_i, b_i \in \lbar{k}.
	\]
Based on multiplicity of its zeros, a non-zero binary quartic form $f$ can be assigned one of the following types:
	\[
		(1, 1, 1, 1), \quad (1, 1, 2), \quad (1, 3), \quad	(2, 2), \quad (4).
	\]
For instance, type $(1, 1, 1, 1)$ includes those binary quartic forms with no multiple root, while type $(1, 1, 2)$ includes those with exactly one double root, and so on. It is clear that if two geometric points $f, g\in V(\lbar{k})-\{0\}$ are conjugate, then they have the same type and also have the same invariants $a$ and $b$. The converse is also true.
	
\begin{prop} \label{prop:orbits_on_fibers}
	In each geometric fiber of $\pi: V\to S$, 	$G$ acts transitively on the set of geometric points of a given type. In other words, if $f, g\in V(\lbar{k})-\{0\}$ have the same invariants $a$ and $b$, and are of the same type, then there exists an element of $h\in G(\lbar{k})$ such that $hf=g$.

  Let $(a, b) \in \bar{k}^2$ be a geometric point of $S$. Then the geometric fiber $V_{(a, b)} = \pi^{-1}(a, b)$ has the following descriptions:
  \begin{enumerate}[\quad (i)]
		\setlength{\itemsep}{0pt}
		\setlength{\parskip}{0pt}
    \item  If $\Delta(a, b)\neq 0$, $V_{(a, b)}$ has precisely one orbit, and it is of type $(1, 1, 1, 1)$.
    \item  If $\Delta(a, b) = 0$ but $(a, b)\neq (0, 0)$, $V_{(a, b)}$ has two orbits, which are of types $(1, 1, 2)$ and $(2, 2)$.
    \item  Finally, $V_{(0, 0)}$ has three orbits, which are of types $(1, 3), (4)$ and $f = 0$.
  \end{enumerate}
\end{prop}
	
A non-zero binary quartic form $f\in V(\bar k)$ is said to be {\em stable} if it has at least one single zero, or in other words if it is of one of the types $(1,1,1,1),(1,1,2)$ or $(1,3)$. We will first treat the stable case.	
	
\begin{prop}\label{prop:stable-orbit}
Let $f\in V(\bar k)$ be a stable binary quartic form. Then there exists $h\in G(\bar k)$ such that 
\[
  hf=y(x^3+axy^2+by^3)
\]
where $a=a(f)$ and $b=b(f)$.
\end{prop}

\begin{proof}
Let ${\mathbb P}^1$ be the projective line with projective coordinate $[x:y]$, where $\infty$ is defined by the equation $y=0$. By conjugation, we can assume that $f$ has a single zero at $\infty$. In other words, it has the form
\[
  f=y(c_0 x^3+c_1 x^2 y + c_2 xy^2 + c_3 y^3)
\]
with $c_0\in\bar k^\times$ and $c_1,c_2,c_3\in \bar k$. The subgroup of upper triangular matrix in $G$ stabilizes $\infty\in \mathbb P^1$. Its action allows us to bring the cubic factor into the form $x^3+axy^2+by^3$ provided that $\car k\neq 3$. We can then check that $a=a(f)$ and $b=b(f)$ on the form $y(x^3+axy^2+by^3)$.  
\end{proof}
	
\begin{proof} (of proposition \ref{prop:orbits_on_fibers}) The case of stable orbits is  already settled by proposition \ref{prop:stable-orbit}. Indeed, since any stable binary quartic form $f$ of invariant $(a,b)$ is conjugate to the polynomial  $y(x^3+axy^2+by^3)$, two stable binary quartic forms of the same invariant  $(a,b)$ are conjugate. Also $\Delta(a,b)\neq 0$ if and only if the cubic  polynomial $x^3+ax^2+b$ have three distinct zeros. If $\Delta(a,b)=0$, it has at least a double zero, and furthermore, it has a triple zero if and only if $(a,b)=(0, 0)$.

We next consider the case of a quartic form $f$ type $(2,2)$. By using the action of $G$ we can assume that $f$ has double zeros at $0$ and $\infty$. In other words, $f$ is of the form $f=cx^2 y^2$ with $c\neq 0$. We observe that in this case, the invariants $a(f)=-c^2/3$ and $b(f)=2c^3/27$ completely determine $c$, and hence $f$, assuming that the characteristic of $k$ is not $2$ nor $3$. 

We finally consider the case of a quartic form $f$ of type $(4)$. By using the action of $G$ we can assume that $f$ has quadruple zero at $\infty$. In other  words, $f$ is of the form $f=cy^4$ with $c\neq 0$. It is then easy to exhibit a diagonal two by two matrix $h$ such that $hf=y^4$.
\end{proof}

Let $I$ be the universal stabilizer of the action of $G$ on $V$, that is
\[ 
    I = (G\times_S V) \times_{V\times_S V} V, \teq \label{stabilizer}
\]
where $G\times_S V \to V\times_S V$ is defined by $(g, v) \mapsto (v, gv)$ and 
$V\to V\times_S V$ is the diagonal map. This is a group scheme over $V$ 
whose Lie algebra can be described as follows.	
	
	\begin{prop} \label{prop:infinitesimal_stab}
	The infinitesimal stabilizers of the action of $\mathfrak{g}= \Lie(G)$ on $V$ are as follows:
	\begin{enumerate}[\quad (i)]
		\setlength{\itemsep}{0pt}
		\setlength{\parskip}{0pt}
		\item Trivial for points of stable types $(1,1,1,1)$, $(1,1,2)$ and $(1,3)$,
		\item One-dimensional for points of types $(2,2)$ and $(4)$,
		\item All of $\mathfrak{g}$ for the point $f=0$.
	\end{enumerate}
	\end{prop}
	\begin{proof}
The action of $\mathfrak g=\Lie(G)$ on $V$ can be identified with the representation $\Sym^4 \std$ of $\fsl_2$. Let us consider a pair $(X,f)\in \fsl_2\times V$ with $X\neq 0$, $f\neq 0$ but $Xf=0$. Since $X\neq 0$, it is either regular semi-simple or regular nilpotent. 

If $X$ is regular semi-simple, after conjugation by an element $h\in G$, it has the form  
        $$
		X=\mtrix{
				a & 0 \\ 
				0 & -a
			}.
		$$
In this case, $f$ has to be a multiple of $x^2y^2$. In other words, $f$ is of type  $(2,2)$. Conversely, if $f$ is of type $(2,2)$, it is conjugate to a quartic polynomial of the type $c x^2 y^2$ with $c\neq 0$ whose infinitesimal centralizer is the space of diagonal matrices in $\fsl_2$. 
		
If $X$ is regular nilpotent, after 
conjugation by an element $h\in G$, it has the form  
        \[ 
		      X=\mtrix{
				  0 & 1 \\ 
				  0 & 0}. \teq \label{nilpotent-matrix}
        \]
The space of $f$ annihilated by $X$ is generated by $y^4$. In other words, $f$ is of type $(4)$. Conversely, if $f$ is of type $(4)$, it is conjugate to $y^4$. Its infinitesimal centralizer is a one-dimensional space of matrices generated by a non-zero nilpotent matrix \eqref{nilpotent-matrix}.
	\end{proof}

We can compute explicitly the geometric stabilizers in stable orbits. Since there  is no infinitessimal stabilizer by proposition~\ref{prop:infinitesimal_stab}, it suffices to determine the $\lbar{k}$-points of $I_f$ for a given stable binary quartic form.
	
	\begin{prop} \label{prop:stabilizer}
	If $f \in V(\lbar{k})$ is of type $(1, 1, 1, 1)$, $(1, 1, 2)$ and $(1, 3)$, then $I_f$ is isomorphic to $\mathbb{Z}/2\mathbb{Z}\times\mathbb{Z}/2\mathbb{Z}$, $\mathbb{Z}/2\mathbb{Z}$ and $0$, respectively.
	\end{prop}
	\begin{proof}
    The case where $f$ is of type $(1, 1, 1, 1)$ is postponed to proposition~\ref{prop:E[2]}. 
		
	  If $f$ is of type $(1, 1, 2)$, by the action of $G$, we can assume that $f = cxy(x-y)^2$. Thus, each element in the stabilizer of $f$ must stabilize the multiset $\{0, \infty, 1^{(2)}\}$. If $h\in I_f$ then either it stabilizes all three points $\{0,1,\infty\}$, or it exchanges $0, \infty$ and stabilizes $1$. Since an element of $G$ is completely determined by its action on three distinct points on $\mathbb{P}^1$, the stabilizer in this case is at most $\mathbb{Z}/2$. A direct calculation shows that it is equal to $\mathbb{Z}/2 \mathbb{Z}$.
		
		For type $(1, 3)$, as above, we can assume that $f = cx^3 y$. Each element in the stabilizer of $f$  must stabilize the multiset $\{0^{(3)}, \infty\}$, which means it stabilizes both $0$ and $\infty$. An element of $G$ fixing both points $0$ and $\infty$ has to lie in the diagonal torus. Now the diagonal torus acts on $x^3y$ by scalar multiplication, and only scalar matrices stabilizes $x^3 y$.
	\end{proof}	
	
	\begin{prop} \label{prop:Vreg_open_dense}
	The union of orbits of stable types $(1, 1, 1, 1), (1, 1, 2)$ and $(1, 3)$ is a dense open subset $V^\reg$ of $V$, which contains the image of the Weierstrass section $s: S\to V$. The restriction of $\pi:V\to S$ to $V^\reg$ is smooth. Moreover, the restriction of the stabilizer group scheme $I$ to $V^\reg$ is \etale.
	\end{prop}
	
	\begin{proof}
	The first two assertions follow directly from 
proposition~\ref{prop:infinitesimal_stab} and~\ref{prop:stable-orbit} above. We derive from~\ref{prop:infinitesimal_stab}, that the morphism $m: G\times S \to V$, defined by restricting the action morphism to the Weierstrass section, is \'etale. By proposition \ref{prop:orbits_on_fibers}, the image of this map is $V^\reg$. We infer that the restriction of $\pi$ to $V^\reg$ is smooth. Moreover, the morphism $G_S \times_S V^\reg \to V^\reg\times_S V^\reg$ defined by $(g, v) \mapsto (v, gv)$ is \'{e}tale, and in particular, the restriction of $I$ to $V^\reg$ is an \etale{} group scheme.	
\end{proof}

\begin{cor} \label{prop:I_S}
There exists a unique group scheme $I_S$ over $S$ equipped with a $G$-equivariant isomorphism $\pi^* I_S \to I$ over $V^\reg$. There is a $\Gm$-equivariant  isomorphism $[BI_S]=[V^\reg/G]$ where $BI_S$ is the relative classifying stack of $I_S$ over $S$. 
\end{cor}
\begin{proof}
  The group scheme $I_S$ is obtained by descending $I$ along $\pi|_{V^\reg}$. The descent datum is obtained using the conjugating action of $G$ on $I$ and the fact that $I$ is abelian. The rest of the corollary is a formal consequence of what we have established so far.
\end{proof}

\section{Elliptic curves}
	
	The relation between elliptic curves and invariant geometry of binary quartic forms has been discovered since 19th century by Cayley and Hermite, and later stated with precision by Weil \cite{Weil:1954tv}.  
		
	\subsection{Jacobian of genus one curves}

Let $D_V$ be the family of arithmetic genus one curves defined over $V$ by the equation $z^2 = f(x, y)$ where $f$ varies over all binary quartic forms. It is constructed by the following cartesian diagram: 
\[
	\xymatrix{
		D_V \ar[d] \ar[r] & \mathcal{O}_{\mathbb{P}^1_V}(2) \ar[d]^{(-)^2} \\
		\mathbb{P}^1_V \ar@/^1pc/[r]^f \ar[d] & \mathcal{O}_{\mathbb{P}^1_V}(4) \ar[l] \\
		V
	}
	\teq \label{diag:universal_quartic}
\]
where $f$ is the universal binary quartic form, and $(-)^2: \mathcal{O}_{\mathbb{P}^1_V}(2) \to \mathcal{O}_{\mathbb{P}^1_V}(4)$ is the squaring map.

\begin{lem}
  If $f\in V(\bar k)-\{0\}$, $D_f$ is reduced. If $f\in V^\reg(\bar k)$, $D_f$ is integral. 
\end{lem}

\begin{proof}
For every $f\in V$, the curve $D_f$ is defined on the ruled surface $\mathcal{O}_{\mathbb{P}^1_V}(2)$ by one single equation. For $f\neq 0$, it is generically reduced and thus reduced. If moreover $f\in V^\reg(\bar k)$, the restriction of $D_f$ over the formal completion of $\mathbb P^1$ at a simple zero of $f$ is an irreducible covering of this formal disc. We deduce that $D_f$ is irreducible for every $f\in V^\reg(\bar k)$. Since $D_f$ is reduced and irreducible, it is integral.
\end{proof}

Let $D^\reg$ be the restriction of $D$ to $V^\reg$. We can now apply the representability of the relative Picard functor and obtain 
the scheme $\Pic_{D^\reg/V^\reg}$ locally of finite type over $V^\reg$. The Jacobian $E_{V^\reg} = \Pic^0_{D^\reg/V^\reg}$ over $V^\reg$ is defined to be the component classifying line bundles of degree $0$. The smooth locus $D^\sm$ of $D^\reg\to V^\reg$ can be identified with $\Pic^1_{D^\reg/V^\reg}$, which is the component classifying line bundles of degree $1$. In particular, $D^\sm$ is an $E$-torsor over $V^\reg$.

One can easily check that if $f\in V^\reg(\bar k)$ is a binary quartic form of one of the types $(1, 1, 1, 1)$, $(1, 1, 2)$ and $(1, 3)$, then $E_f$ is an elliptic curve, $\Gm$ and $\Ga$ respectively. In the first case, $D_f$ is a smooth genus one curve acted on simply transitively by the elliptic curve $E_f$. In the two latter cases, $D_f$ is a rational curve, with nodal or cuspidal singularity respectively, acted on by $E_f$.

Over $S$, the universal Weierstrass curve $E_S$ is defined to be the closed subscheme of $\mathbb{P}^2_S$ given by the equation:
$$
	z^2y = x^3 + axy^2 + by^3.
$$
Following Cayley and  Hermite, Weil proved in~\cite{Weil:1954tv} that for every binary quartic form $f\in V^\reg(\bar k)$ of type $(1,1,1,1)$ of invariant $(a,b)\in S(\bar k)$,  there is a canonical isomorphism $E_f=E_{a,b}$. His proof can be extended to the regular locus so that we have a canonical isomorphism
\[
  E_{V^\reg} \to E_S \times_S {V^\reg}. \teq \label{Weil-Cayley-Hermite}
\]
We remark that this isomorphism can be made naturally $G\times \Gm$-equivariant compatible  with the action of $G\times \Gm$ on $V$ given by the formula \eqref{Gm-action}. 
	
	\subsection{Centralizer and $2$-torsion of elliptic curves}
	In this subsection, we will present what we see as an important link between the arithmetic of elliptic curves and invariant geometry of binary quartic forms. Recall that over $S$, formula \eqref{stabilizer} defines the stabilizer group scheme $I$, which is quasi-finite and \'etale over $V^\reg$.  
	
\begin{prop} \label{prop:E[2]}
Over $V^\reg$, there is a canonical isomorphism
		$$
			I|_{V^\reg} \cong E[2] |_{V^\reg}.
		$$
\end{prop}

\begin{proof} By construction \eqref{diag:universal_quartic}, $G$ acts on the family of arithmetic genus one curve $D$ over $V^\reg$. This induces an action of $G$ on the Jacobian $E$ of $D$. For every $f\in V^\reg$, the stabilizer $I_f$ acts on the genus one curve $D_f$ and its Jacobian $E_v$. It follows from the Cayley-Hermite-Weil theorem \eqref{Weil-Cayley-Hermite} that $I_f$ acts trivially on $E_f$. 

As our construction is functorial, if $h\in I_f$, $d\in D_f^{\rm sm}$ and $e\in E_f$, we have 
$$h(ed)= h(e) h(d)$$
where $ed$ denotes the action of $E_f$ on $D_f^{\rm sm}$. Since $I_f$ acts trivially on $E_f$, the above equality implies that the action of $I_f$ and $E_f$ on $D_f^{\rm sm}$ commute. As $D_f^{\rm sm}$ is a torsor under the action of $E_f$, this gives rise to a homomorphism 
\[
  I_f\to E_f \teq \label{I-to-E}
\]
through which the action of $I_f$ on $D_f^{\rm sm}$ factors.

We will first prove that  the homomorphism \eqref{I-to-E} factors through the subgroup $E_f[2]$ of $2$-torsions of $E_f$. It suffices to prove this for $f$ of type $(1,1,1,1)$, since the general case follows by flatness. Let $R_f$ denote the ramification locus of $D_f$ over ${\mathbb P}^1$. One can check that $E_f[2]$ acts simply transitively on $R_f$ and this action commutes with the action of $I_f$. This gives rise to a homomorphism $I_f\to E_f[2]$ through which~\eqref{I-to-E} factors. 

For both $I$ and $E[2]$ are \'etale group schemes over $V^\reg$, in order to prove that $I\to E[2]$ is an isomorphism, it is enough to check that it induces a bijection on geometric points over each $f\in V^\reg(\bar k)$. 

Let $f\in V^\reg(\bar k)$ be of type $(1,1,1,1)$. Let $h\in I_f$ be an element with trivial image in $E_f[2]$. In this case, $f$ fixes all the four ramifications points of $D_f$. In other words, it fixes the four zeros of $f$, which implies that $h=1$ since ${\rm PGL}_2$ acts sharply $3$-transitive on the projective line. It follows that the homomorphism $I_f\to E_f[2]$ is injective. It must also be  surjective, for both groups $I_f$ and $E_f[2]$ have $4$ elements.

For type $(1,1,2)$, this is an explicit calculation for nodal rational curve as in proposition \ref{prop:stabilizer}. Finally, for type $(1,3)$, there is nothing to be proved, since both groups $I_f$ and $E_f[2]$ are trivial. \end{proof}

The isomorphism $I\to E[2]$ over $V^\reg$ is by construction $G$-equivariant. It descends to an isomorphism of group schemes $I_S \to E_S[2]$ over $S$, where $I_S$ is defined in proposition  \ref{prop:I_S} and $E_S$ in \eqref{Weil-Cayley-Hermite}. It follows from proposition  \eqref{prop:I_S} that there exists a $\Gm$-equivariant isomorphism
\[
  BE_S[2]=[V^\reg/G]. \teq \label{BE2}
\]

\subsection{Link to 2-Selmer groups}	

Recall that $C$ is a smooth, projective and geometrically connected curve over $k$. We will denote $K=k(C)$ the field of rational functions of $C$ and $K_v$ 
its completion at a closed point $v\in |C|$.

For each morphism $\alpha:C\to [S/\Gm]$ we have a family of Weierstrass curve $E_\alpha=\alpha^*E_S$. The groupoid of maps $\beta:C\to [BI_S/\Gm]$ over $\alpha$ is by definition the groupoid of $I_\alpha$-torsors over $E$ where $I_\alpha=\alpha^* I_S$. We will show in this section that there is a closed connection between this groupoid and the 2-Selmer group of the generic fiber $E_{\alpha,K}$ of $E_{\alpha}$. We recall that for each  elliptic curve $E$ defined over $K$, the 2-Selmer group of $E$ is defined as the kernel of the homomorphism:
\[
  {\rm Sel_2}(E)={\rm ker}(H^1(K, E[2]) \to \prod_{v\in |C|} H^1(K_v, E)).
\]
We will write $\Sel_2(E_\alpha)$ instead of $\Sel_2(E_{\alpha,K})$ as this shorthand doesn't cause any confusion.

The \etale{} cohomology group $H^1(C,I_\alpha)$ is naturally identified with the group of isomorphism classes of $I_\alpha$-torsors over $E$. By restriction to the generic fiber of $C$, we obtain a homomorphism
\[
  H^1(C,I_\alpha) \to H^1(K, I_{\alpha})=H^1(K, E_{\alpha}[2]).\teq \label{restriction-to-generic}
\]
	
\begin{prop} \label{prop:Lang_theorem}
The homomorphism \eqref{restriction-to-generic} factors through the 2-Selmer group ${\rm Sel}_2(E_{\alpha})$.
\end{prop}

\begin{proof}
		We have the following commutative diagram for each $v\in |C|$:
		\[
		\xymatrix{
			H^1(C, I_\alpha) \ar[d] \ar[r] & H^1(K, I_\alpha) \ar[d] \\
			H^1(\Spec \mathcal{O}_v, I_\alpha) \ar[d] \ar[r] & H^1(K_v, I_\alpha) \ar[d] \\
			H^1(\Spec \mathcal{O}_v, E_\alpha) \ar[r] & H^1(K_v, E_\alpha).
		}
		\]
But by Lang's theorem, we know that $H^1(\Spec \mathcal{O}_v, E_\alpha) = 0$ since $E_S$ has connected fibers. It follows that the composition map 
\[
  H^1(C, I_\alpha) \to H^1(K_v, E_\alpha)
\]
is trivial for all $v\in |C|$. The lemma follows.
\end{proof}

As a corollary, we obtain a natural map
\[
	\rho_\alpha: H^1(C,I_\alpha)=H^1(C, E_\alpha[2]) \to \Sel_2(E_\alpha) \teq \label{H1I-Sel2}
\]
for all maps $\alpha:C\to [S/\Gm]$ whose image is not contained in the discriminant locus.
	
\begin{prop} \label{prop:transversal_Sel_torsor}
If $\alpha: C\to [S/\Gm]$ is transversal to the discriminant locus in the sense of \ref{transversal}, then the homomorphism $\rho_\alpha: H^1(C, H^1(C, E_\alpha[2])) \to \Sel_2(E_\alpha)$ is an isomorphism.
\end{prop}
	
\begin{proof} The assumption $\alpha: C\to [S/\Gm]$ is transversal to the discriminant locus implies that $E_\alpha/C$ is a smooth group scheme with elliptic or multiplicative fibers, which is the global \Neron{} model of its generic fiber. 

Let $v$ be a geometric point of $C$ such that $c(v)$ lies in the discriminant locus. We denote $C_v$ the completion of $C\otimes_k \bar k$ at $v$, $\Spec(K_v)$ the generic point of $C_v$, and $I_v=\Gal(K_v)$. The transversality implies that $\Delta$ vanishes at $v$ to order 1. Using the description of the Tate curve, we know that
$$(E_\alpha(K_v)[2])^{I_v}=\mathbb Z/2\mathbb Z.$$
Geometrically, this means that over $C_v$, the \'etale group scheme $E_\alpha[2] $ is exactly the \'etale locus in its normalization over $C_v$. We deduce that globally, $E_\alpha[2]$ is exactly the \'etale locus in its normalization over $C$. 

This observation will allow us to prove the injectivity of the map 
\[
  H^1(C,E_\alpha[2])   	\to H^1(K,E_\alpha[2]).
\]
Indeed, let $T$ be an $E_\alpha[2]$-torsor over $C$. We will prove that $T$ is uniquely determined by its generic fiber. We observe that for every geometric point $v$ of $C$, the restriction of $T$ to the formal disc $C_v$ is isomorphic to the restriction of $E_\alpha[2]$. As $E_\alpha[2]$, $T$ restricted to $C_v$ is exactly the \'etale locus of its normalization over $C_v$. Hence, globally, $T$ can also be identified with the \'etale locus of its normalization over $C$. This means that we can reconstruct $T$ by removing the ramification locus from the normalization of its generic fiber. This proves the injectivity of $\rho_\alpha$.

We will now prove that $\rho_\alpha$ is surjective. Let $T_K$ be an $E_\alpha[2]$-torsor over $K$ whose isomorphism class lies in $\Sel_2(E_\alpha)$. We will show that the Selmer condition implies that $T$ can be extended as an $I$-torsor over $C$. We first spread $T$ to a $E_\alpha[2]$-torsor defined over some nonempty open subset $U$ of $C$. After that, we only need to prove that $T$ can be extended to a $E_\alpha[2]$-torsor over the formal discs $C_v$ around the remaining points, and thus we are reduced to a local problem.

The Selmer condition at $v$ implies that the class of $T$ in $H^1(K_v, E_\alpha[2])$ lies in the image of $E_\alpha(K_v)/2E_\alpha(K_v)$. There exists a point $x\in E_\alpha (K_v)$ such that the torsor $T_{K_v}$ fits in a cartesian diagram:
		\[
		\xymatrix{
			T_{K_v}	\ar[d] \ar[r] & E_{\alpha,K_v} \ar[d]^{\cdot 2} \\
			\Spec K_v \ar[r]^>>>>>x & E_{\alpha,K_v}
		}
		\]
Since $E_{\alpha,C_v}$ is the \Neron{} model of $E_{\alpha,K_v}$, the $K_v$-point $x$ of $E_\alpha$ can be extended as a $C_v$-point $\tilde x$. We can now extend the $E_{K_v}$-torsor $T_{K_v}$ to a $E_{C_v}$-torsor by forming the cartesian diagram
$$
		\xymatrix{
			T_{C_v}	\ar[d] \ar[r] & E_{\alpha,C_v} \ar[d]^{\cdot 2} \\
			C_v \ar[r]^>>>>>x & E_{\alpha,C_v}
		}
		$$
This completes the proof of surjectivity of $\rho_\alpha$. 
\end{proof}
	
In the case where $\alpha$ is not transversal to the discriminant locus, it can happen that the homomorphism $\rho_\alpha$ is neither surjective nor injective. Nevertheless, we can compare sizes of $\Sel_2(E_{\alpha})$ and $H^1(C,I_\alpha)$.
	
	\begin{prop} \label{prop:inequalities}
		Let $\alpha: C\to [S/\Gm]$ and suppose that the generic fiber of $E_\alpha$ is an elliptic curve. Then
		\[
		\begin{cases}
			|\Sel_2(E_\alpha)| \leq |H^1(C, I_\alpha)|, &\text{when } E_\alpha[2](K) = 0, \\
			|\Sel_2(E_\alpha)| \leq 4|H^1(C, I_\alpha)|, &\text{otherwise.}
		\end{cases}
		\]
	\end{prop}
	\begin{proof}
		From the proof of proposition~\ref{prop:transversal_Sel_torsor}, we always have
		\[
			|\Sel_2(E_\alpha)| \leq |H^1(C, \mathcal{E}[2])|,
		\]
    where $\mathcal{E}$ is the \Neron{} model of the generic fiber of $E_\alpha$ over $C$, since we can always lift a Selmer class to a torsor of $\mathcal{E}[2]$ over $C$. Note that in the proof of proposition~\ref{prop:transversal_Sel_torsor}, we lift the Selmer class to an $E[2]$-torsor over $C$, exploiting the isomorphism $E \cong \mathcal{E}$ in the transversal situation.
		
		From the short exact sequence of group schemes over $C$
		$$
		\xymatrix{
			0 \ar[r] & E_\alpha[2] \ar[r] & \mathcal{E}[2] \ar[r] & Q \ar[r] & 0,
		}
		$$
where $Q$ is a skyscraper sheaf, we have the following long exact sequence
		$$
		\xymatrix@C=1.3em{
			0 \ar[r] & H^0(E_\alpha[2]) \ar[r] & H^0(\mathcal{E}[2]) \ar[r] & H^0(Q) \ar[r] & H^1(E_\alpha[2]) \ar[r] & H^1(\mathcal{E}[2]) \ar[r] & H^1(Q) \ar[r] & L \ar[r] & 0.
		}
		$$
where $L$ is the kernel of the map $H^2(E_\alpha[2])\to H^2(\mathcal E[2])$.		
		
Since $Q$ is a skyscraper sheaf, its cohomology groups are direct sums of Galois cohomology groups of finite fields. It follows that
		$$
		|H^0(Q)| = |H^1(Q)|.
		$$
Using multiplicative Euler characteristic, combined with the fact that 
$$|H^0(E_\alpha[2])|=|H^0(\mathcal E[2])| =1$$ 
under the assumption $E_\alpha[2](K) = 0$, or 
$$|H^0(\mathcal{E}[2])|/|H^0(E_\alpha[2])| \leq 4,$$
without this assumption, we get the desired inequality. \end{proof}
	
	\subsection{The geometric setup} \label{subsec:geometric_setup}
	We can now define the moduli spaces $\mathcal{M}_\mathcal{L}$ and $\mathcal{A}_\mathcal{L}$ promised at the end of section~\ref{sec:Weierstrass_models}. First, we denote
	\begin{align*}
		\mathcal{M} &= \Hom(C, [BI_S/\Gm]) \\
		\mathcal{A} &= \Hom(C, [S/\Gm]).
	\end{align*}
We clearly have a map $\mathcal{M} \to \mathcal{A}$, compatible with the natural map to $\Bun_{\Gm} = \Hom(C, B\Gm)$. 

For a given line bundle $\mathcal{L} \in \Bun_\Gm(k)$ over $C$, we denote  $\mathcal{M}_\mathcal{L}$ and $\mathcal{A}_\mathcal{L}$ the fiber of $\mathcal{M}$ and $\mathcal{A}$ over $\mathcal{L}$. The space $\mathcal A_{\mathcal L}$ classifies family of Weierstrass curves of Hodge bundle $\mathcal L$. For a given $\alpha:C\to [S/\Gm]$, we denote $E_\alpha=\alpha^* E$ the induced family of Weierstrass elliptic curves and $E_\alpha[2]$ its 2-torsion subgroup. The fiber of $\mathcal M\to \mathcal A$ over $\alpha$, classifying $E_\alpha[2]$-torsor over $C$ is our replacement for the Selmer group $\Sel_2(E_{\alpha})$, for as shown in proposition \ref{prop:transversal_Sel_torsor}, there is a canonical isomorphism $H^1(C,E_\alpha[2]) \to {\rm Sel}_2(E_{\alpha})$ in case $\alpha$ is transversal to the discriminant locus, anh otherwise we have the inequality in 
proposition  \ref{prop:inequalities}.

Even though it is not easy to count points on $\Hom(C, [BI_S/\Gm])$ directly, the invariant theory of binary quartic forms allows us to represent $\mathcal M$ by yet another way. Namely, \eqref{BE2} induces an isomorphism:
\[
  \mathcal{M}= \Hom(C, [V^\reg/ G\times \Gm]).
\]
By definition, a $k$-point of $\mathcal{M}$ consists of a triple $(\mathcal{E}, \mathcal{L}, \alpha)$, where $\mathcal{E}$ is a $G$-torsor, $\mathcal{L}$ a line bundle, and $\alpha$ a section of $V(\mathcal{E}, \mathcal{L})^\reg = (V^\reg\times^G \mathcal{E}) \otimes \mathcal{L}^{\otimes 2}$. This new presentation is thus very convenient for counting points, since we are essentially counting sections of the vector bundle $V(\mathcal{E}, \mathcal{L}) = (V\times^G \mathcal{E})\otimes \mathcal{L}^{\otimes 2}$ satisfying some condition.

This suggests that instead of counting points on $\mathcal{M}$, we should count points on
\[
  \mathcal M'=\Hom(C, [V/ G\times \Gm]).
\]
and study the ratio between the two numbers. The $k$-points on $\mathcal{M}'$ are of course those triples $(\mathcal{E}, \mathcal{L}, \alpha)$ where $\mathcal{E}, \mathcal{L}$ are as above, and $\alpha$ is a section of $V(\mathcal{E}, \mathcal{L})$.

However, one needs to pay attention to the fact that for any line bundle $\mathcal L$, the number of $k$-points on $\mathcal M'_{\mathcal L}$ is infinite. In order to make sense of the ratio, one fix a $G$-bundle $\mathcal E$, and calculate the ratio 
\[
  \frac{|\mathcal M_{\mathcal E,\mathcal L}(k)|}{|\mathcal M'_{\mathcal E,\mathcal L}(k)|}
\]
as $\deg(\mathcal L) \to \infty$ while $\mathcal E$ being fixed. This ratio calculation will be performed in the next section following some ideas of Poonen.

	\section{On density}
	\subsection{Poonen's results}
	
	In this section, we will prove a density result that allows us to compute the difference between the number of sections to the regular part and the number of all sections. As the main ideas are already presented in~\cite{poonen_squarefree_2003}, we will only indicate necessary modifications in the proof. 
	
\begin{prop} \label{prop:Poonen_Density}
		Let $C$ be a smooth projective curve over $\mathbb{F}_q$, $\mathcal{E}$ a vector bundle over $C$ of rank $n$. Let $X\subset \mathcal{E}$ be a locally closed $\Gm$-stable subscheme of codimension at least 2 whose fiber at every point $v\in C$, $X_v \subset \mathcal{E}_v$ is also of codimension at least 2. Then the ratio 
\[
  \mu(X,\mathcal L)=\frac{|\{s\in \Gamma(C, \mathcal{E}\otimes \mathcal{L}): s\text{ avoids } X\otimes \mathcal{L}\}|}{|\Gamma(C, \mathcal{E} \otimes \mathcal{L})|}
\]
as $\deg(\mathcal L)\to \infty$, tends to the limit
\[
  \mu(X) := \lim_{\deg \mathcal{L} \to \infty} \mu(X,\mathcal L) = \prod_{v \in |C|} \left(1-\frac{c_v}{|k(v)|^n}\right),
\]
where $c_v = |X_v(k(v))|$, with $k(v)$ denoting the residue field at $v$.
\end{prop}
	
	The main point of this result is that the density can be computed as the product of local densities, which are the factors in the product on the RHS of the formula above. Before starting the proof, we first prove the following lemma.
	
	\begin{lem} \label{lem:deg_uniform_infty}
	 Let $C$ be a smooth projective curve over $k$. There exists a finite set $S \subset |C|$ and a number $n$ such that for all line bundles $\mathcal{L}$ with $\deg \mathcal{L} > n$, there exists an effective divisor $D$ supported on $S$ such that $\mathcal{L} \cong \mathcal{O}_C(D)$. Moreover, we can choose $D_\mathcal{L} = \sum_{v\in S} a_v(\mathcal{L}) v$ for each $\mathcal{L}$ such that as $\deg\mathcal{L}$ goes to $\infty$, so does $a_v(\mathcal{L})$ for each $v\in S$.
	\end{lem}
	
	\begin{proof}
We start with $m$ distinct points $Q_1,\ldots,Q_m \in |C|$ with $m$ being a big enough integer such that $\mathcal{L}(\sum_{j=1}^m Q_j)$ has non zero global sections for all line bundles $\mathcal{L} \in \Pic^0_{C/\mathbb{F}_q}(\mathbb{F}_q)$. It follows that every line bundle $\mathcal{L} \in \Pic^0_{C/\mathbb{F}_q}(\mathbb{F}_q)$ can be written as 
\[
  \mathcal L=\mathcal{O}_C\left(\sum_i P_i -\sum_{j=1}^m Q_j \right). \teq \label{PQ}
\]
Since $\Pic^0_{C/\mathbb{F}_q}(\mathbb{F}_q)$ is a finite set, there are finitely many points $P_i$ that may appear in \eqref{PQ}. We let $S$ be the union of all the $Q_j$ and $P_i$ appearing above. 

We also suppose that the points $Q_1,\ldots,Q_m$ have been chosen such that 
their degrees are relatively prime. It that case the monoid generated by $
\deg(Q_1),\ldots,\deg(Q_m)$ will contain all integers $d$ big enough. That is, there exists $N$ such that for all $d>N$, we can write 
$d=\sum_{j=1}^m d_i \deg(Q_i)$ 
with $d_i$ being positive integers. We can also choose the integers $d_j$ in 
such a way that each  $d_j\to \infty$ as $d\to \infty$.

Let $\mathcal{L} \in \Pic^d_{C/\mathbb{F}_q}(\mathbb{F}_q)$. If $d > N$, then we can write
		\[
			\mathcal{L} \cong \mathcal{O}\left(\sum_{j=1}^m d_i Q_i\right) \otimes \mathcal{L}', 
		\]
and $\deg \mathcal{L}' = 0$. Then by using \eqref{PQ}, we have 
    \[
			\mathcal{L} \cong \mathcal{O}\left(\sum P_i + \sum_{j=1}^m (d_j-1)Q_j\right),
		\]
where $P_i, Q_i \in S$. 
		
The last part of the lemma can be proved by an obvious modification of the argument above.
	\end{proof}
	
	\begin{rmk}
    From the proof of the lemma, we see at once that the set $S$ can always be made arbitrarily large.
  \end{rmk}
	
	Following~\cite[theorem 3.1]{poonen_squarefree_2003}, we will prove proposition~\ref{prop:Poonen_Density} by showing that we can compute the density as the limit of a finite product of densities over closed points where the sizes of the residue fields are bounded. The following lemma enables us to do so.

\begin{lem} \label{lem:density_key_lem}
Let $C, \mathcal{E}$ and $X$ be as in proposition~\ref{prop:Poonen_Density}. For each $M > 0$ we define
		$$
			\mathcal{Q}_{M, \mathcal{L}} = \{s\in \Gamma(X, \mathcal{E}\otimes\mathcal{L}): \exists v \in |C|, |k(v)| \geq M \text{ and } s_x \in X_x\}.
		$$
		Then
		$$
			\lim_{M\to \infty} \limsup_{\deg \mathcal{L} \to \infty} \frac{|\mathcal{Q}_{M, \mathcal{L}}|}{|\Gamma(X, \mathcal{E}\otimes\mathcal{L})|} = 0.
		$$
	\end{lem}
	\begin{proof}
		This statement is more or less a restatement of what is already proved in the first part of the proof of~\cite[theorem 8.1]{poonen_squarefree_2003} (see also~\cite[lemma 5.1]{poonen_squarefree_2003}). We will thus only indicate why this is the case.
		
		Since we are only interested in the case where $M \gg 0$, we can throw away as many points of $C$ as we want. We can therefore replace $C$ by any open affine subscheme $C'$ such that $\mathcal{E}$ is free over $C'$. Now, lemma~\ref{lem:deg_uniform_infty} implies that we can choose $C'$ such that our limit has the same form as the limit defined in~\cite[theorem 8.1]{poonen_squarefree_2003}.

    Observe that Poonen proves his limit for the case where $X|_{C'}$ is defined by 2 equations that are generically relative primes. But now, we can conclude by noting that since $X$ is of codimension at least 2, we can find such $f, g$ that both vanish on $X$ (see the proof of~\cite[lemma 5.1]{poonen_squarefree_2003}).
	\end{proof}
	\begin{proof}[Proof of~\ref{prop:Poonen_Density}]
		The proof of~\ref{prop:Poonen_Density} can be carried word for word from the proof of~\cite[theorem 3.1]{poonen_squarefree_2003}, where lemma~\ref{lem:density_key_lem} plays the role of~\cite[lemma 5.1]{poonen_squarefree_2003}. Indeed, if we denote
		$$
		  \mu(X_M) = \lim_{\deg\mathcal{L}\to \infty} \frac{|\{s\in \Gamma(C, \mathcal{E}\otimes \mathcal{L}): s\text{ avoids } X\otimes \mathcal{L} \text{ at all } v\in |C|, |k(v)|<M\}|}{|\Gamma(C, \mathcal{E} \otimes \mathcal{L})|},
		$$
		then lemma~\ref{lem:density_key_lem} implies that
		$$
		  \mu(X) = \lim_{M\to \infty} \mu(X_M).
		$$
		
		Note that the linear map 
		$$
		  \Gamma(C, \mathcal{E}\otimes \mathcal{L}) \to \prod_{\substack{v\in |C| \\ |k(v)|<M}} \mathcal{E}\otimes\mathcal{L}\otimes k(v) \cong \prod_{\substack{v\in |C| \\ |k(v)|<M}} \mathcal{E}\otimes k(v)
		$$
		is surjective when $\deg\mathcal{L} \gg 0$ due to the vanishing of
		$$
		  H^1\left(C, \mathcal{E}\otimes\mathcal{L}\left(-\sum_{\substack{v\in |C| \\ |k(v)|<M}} v\right)\right)
    $$
    when $\deg \mathcal{L}\gg 0$. Thus, we have
		$$
		  \mu(X_M) = \prod_{\substack{v\in |C| \\ |k(v)| < M}} \left(1-\frac{c_v}{|k(v)|^n}\right),
		$$
		where $c_v$ is defined as in proposition~\ref{prop:Poonen_Density}.
	\end{proof}
	
	Using a similar argument, we have the following result also.
	\begin{prop} \label{prop:Poonen_Density_transversal}
		Let $C, \mathcal{E}, X$ as above, and $D \subset \mathcal{E}$ be a subscheme defined by the vanishing of an equation $d: \mathcal{E} \to \mathcal{L}'$, where $\mathcal{L}'$ is a line bundle over $C$. Suppose that $d$ is generically square-free, then
		$$
			\lim_{\deg\mathcal{L} \to \infty} \frac{|\{s\in \Gamma(C, \mathcal{E}\otimes\mathcal{L}): s \in \mathcal{E} \setminus X \text{ and } s \text{ intersects D transversally}\}|}{|\Gamma(C, \mathcal{E}\otimes\mathcal{L})|} = \prod_{v\in |C|} \left(1 - \frac{c_v}{|k(v)|^{2n}}\right),
		$$
where $c_v$ is the number of elements $s$ in $\mathcal{E}\otimes \mathcal{O}_{C, v}/\mathfrak{m}_v^2$ such that $s$ lies in $X\otimes \mathcal{O}_{C, v}/\mathfrak{m}_v^2$ or $d(s) = 0$ in $\mathcal{O}_{C, v}/\mathfrak{m}_v^2$.
	\end{prop}
	\begin{proof}
    The proof of this proposition is almost identical to the one above. As we have seen, all we need to do is to prove the analog of lemma~\ref{lem:density_key_lem} for this case. Observe also that we only need to prove such a lemma for a suitable open affine sub-curve $C'$ which can be chosen such that $\mathcal{E}|_{C'}$ and $\mathcal{L}|_{C'}$ are free. In this case, $d$ is just a generically square-free polynomial with coefficient in $\Gamma(C', \mathcal{O}_{C'})$.

    If $X$ is an empty scheme, this is already done in~\cite[theorem 8.1]{poonen_squarefree_2003}. When $X$ is not empty then we see that the error term is bounded above by the sum of the error term in the case where $X$ is empty and the error term given in~\ref{lem:density_key_lem} above. But since both go to zero as $M$ goes to infinity, we are done.
	\end{proof}
	
	\subsection{Some density computations}
	
	In this subsection, for brevity's sake, we will use $V(\mathcal{E})$ and $V(\mathcal{E})^\reg$ to denote $V(\mathcal{E}, \mathcal{O}_C)$ and $V(\mathcal{E}, \mathcal{O}_C)^\reg$ respectively (see the notation in subsection~\ref{subsec:geometric_setup}), where $\mathcal{E}$ denotes an arbitrary fixed $G$-torsor.
	
	\begin{prop} \label{prop:density_reg}
	The density of $V(\mathcal{E})^\reg$ inside $V(\mathcal{E})$ is $\zeta_C(2)^{-1}$.	\end{prop}
	\begin{proof}
		By proposition~\ref{prop:Poonen_Density}, it suffices to show that the local density at a point $v\in |C|$ of the regular part is $1-|k(v)|^{-2}$. For this, we first count the number of points in the non-regular part. By the classification of different orbits on $V$, we know that a point $f$ in the non-regular part must be of type $(2, 2)$ or $(4)$ or $0$. Thus, we see at once that up to a scalar multiple, $f$ is a square of a quadratic polynomial.
		
		Note that the squaring map (from quadratic to quartic polynomials) is a two to one map, except at the $0$ polynomial. The image of the map is not surjective on the non-regular part, and the missing points are precisely those which are a scale of a point in the image by a non-square element in $k(v)^\times$. Thus, the number of points in the non-regular part is
		$$
		  \frac{|\{\text{non-zero binary quadratic polynomials}\}|}{2} |k(v)^\times/k(v)^{\times 2}| + 1 = \frac{|k(v)|^3-1}{2} 2 + 1 = |k(v)|^3.
		$$
		Thus, the local density of the regular part is
		$$
			\frac{|k(v)|^5-|k(v)|^3}{|k(v)|^5} = 1 - |k(v)|^{-2}.
		$$
	\end{proof}
	
	\begin{prop} \label{prop:density_(a,b)_transversal}
		The density of $(a,b) \in \Gamma(C, \mathcal{L}^{\otimes 4} \oplus \mathcal{L}^{\otimes 6})$ transversal to the discriminant locus among all pairs $(a,b)$ is
		$$
			\prod_{v\in |C|} (1-2|k(v)|^{-2} + |k(v)|^{-3}).
		$$
	\end{prop}
	\begin{proof}
		By proposition~\ref{prop:Poonen_Density_transversal}, it suffices to show that the local density at a point $v\in |C|$ of the transversal part is $1-2|k(v)|^{-2} + |k(v)|^{-3}$. 
		
		Let denote $R=k(v)[\eps]/(\eps^2)$. We observe that $(a,b) \in S(R) = R^2$ is in the transversal part if and only if $\Delta(a,b) \neq 0$ in $R$.
If $(a,b) \in S(R)$, then we denote $(\lbar{a}, \lbar{b}) \in S(k(v))$ the associated $k(v)$-point, by reduction. Observe that $\Delta: \mathbb{A}^2 \to \mathbb{A}^1$ is smooth on $S-\{(0, 0)\}$. In particular, when $(a,b) \in S(R)$ such that $(\lbar{a}, \lbar{b}) \neq (0, 0)$, then the fiber of $T_{(\lbar{a}, \lbar{b})} S \to	T_{\Delta(\lbar{a}, \lbar{b})}$ has dimension exactly one. Thus, the number of non-transversal pairs $(a,b) \in S(R)$ is
		\begin{align*}
			\sum_{\substack{(\lbar{a}, \lbar{b})\neq(0,0)\\ \Delta(\lbar{a}, \lbar{b})=0}} |k(v)| + \sum_{(\lbar{a}, \lbar{b}) = (0, 0)} |k(v)|^2 
			&= |k(v)|(|\Ga(k(v))|-1) + |k(v)|^2 \\
			&= |k(v)|(|k(v)|-1) + |k(v)|^2 \\
			&= 2|k(v)|^2 - |k(v)|.			
		\end{align*}
		Thus, the local density of transversal pairs is
		$$
			\frac{|k(v)|^4-2|k(v)|^2+|k(v)|}{|k(v)|^4} = 1 - 2|k(v)|^{-2} + |k(v)|^{-3},
		$$
		where we have used $|R|^2 = |k(v)|^4$.
	\end{proof}
	
	\begin{prop} \label{prop:density_V_reg_transversal}
		The density of sections in $V(\mathcal{E})$ that are in $V(\mathcal{E})^\reg$ whose associated pair $(a,b)$ is transversal to the discriminant is
		$$
			\prod_{v\in |C|} (1-|k(v)|^{-2})(1-2|k(v)|^{-2} + |k(v)|^{-3}).
		$$
	\end{prop}
	\begin{proof}
		The strategy is similar to what we have done above.	Here, we also compute the complement of the described condition on $V(\mathcal{E})$. As in the previous lemma, we let $v\in |C|$ and $R = k(v)[\eps]/(\eps^2)$. In this computation, for brevity's sake, we denote $k = \mathbb{F}_q = k(v)$, and hence, $q = |k(v)|$. The number of points that fail the described condition is
		\begin{align*}
			&\quad |V^\text{non-reg}(R)| + |V^{\reg, \text{non-transversal}}(R)| \\
			&= \sum_{f\in V^\text{non-reg}(k)} |T_{V, f}(k)| + \sum_{\substack{f\in V^\reg(k) \\ \Delta(f) = 0}} |\ker d\Delta_f(k)| \\
			&= q^3 q^5+ \sum_{\substack{f \in V^\reg(k) \\ a(f) \neq 0, b(f) \neq 0 \\ \Delta(f) = 0}} |\ker d\Delta_f(k)| + \sum_{\substack{f\in V^\reg(k) \\ a(f) = b(f) = 0}} |\ker d\Delta_f(k)|, \teq\label{eq:lem_sum_types}
		\end{align*}
		where $q^3$ comes from the computation made in proposition~\ref{prop:density_reg} above.
		
		Observe that if $f\in V^\reg(k)$, then geometrically, namely, over $\overline{\mathbb{F}_q}$, $f$ is in the same orbit as $y(x^3+a(f) xy^2 + b(f) y^3)$. The condition $\Delta(f) = 0$, then implies that $f$ can only be of type $(1, 1, 2)$ or $(1, 3)$. We see easily that type $(1, 1, 2)$ and type $(1, 3)$ can only occur in the second and third summands, respectively, of~\eqref{eq:lem_sum_types}.
		
		We will now compute the number of $f \in V(k)$ of type $(2, 1, 1)$. We see at once that the double root must be rational and hence, over $k$, we have $f = c(x-ay)^2(x^2+uxy+vy^2)$. Thus, the number of such $f$ can be computed as
		$$
			|\Gm(k)| |\mathbb{P}^1(k)| |\Sym^2 \mathbb{A}^1(k) - \text{diagonal}(k)| = (q-1)(q+1)(q^2-q) = q(q^2-1)(q-1).
		$$
		
		Similarly, the number of $f$ of type $(1, 3)$ can be computed as
		$$
			|\Gm(k)| |\mathbb{P}^1(k)| |\mathbb{A}^1(k)| = (q-1)(q+1)q = q(q^2-1).
		$$
		
		To compute the $|\ker d\Delta_f|$ factors, we note that the map $V^\reg \to S$ is smooth by corollary~\ref{prop:Vreg_open_dense}, and the smooth locus of $\Delta: S \to \mathbb{A}^1$ is precisely $S -\{(0, 0)\}$. This enables us to compute the dimension of $\ker d\Delta_f$, and hence its size, at some point $f\in V^\reg(k)$. Indeed, for type $(1,1,2)$ and $(1,3)$, $|\ker d\Delta_f(k)|$ is $q^3 q = q^4$ and $q^3 q^2 = q^5$ respectively.
		
		Gathering all the results above, we have
		$$
			\eqref{eq:lem_sum_types} = q^8 + q^5(q^2-1)(q-1) + q^6(q-1)(q+1) = 3q^8- q^7	- 2q^6 + q^5.
		$$
		Thus, the number of transversal and regular points in $V(R)$ is
		$$
			q^{10} - 3q^8 + q^7 + 2q^6 - q^5 = q^5(q^2-1)(q^3-2q+1).
		$$
		The local density is thus 
		$$
			(1-q^{-2})(1-2q^{-2}+q^{-3}).
		$$
		as stated.
	\end{proof}
	
A similar computations and arguments as above will give us the following results.
	
\begin{prop} \label{prop:density_minimal_A_B}
		The density of sections in $S$ that are minimal is $\zeta_C(10)^{-1}$.
\end{prop}
	
	\begin{prop} \label{prop:density_V_reg_minimal}
		The density of sections in $V(\mathcal{E})$ that are in $V(\mathcal{E})^\reg$ and whose associated invariant $(a,b)$ is minimal, is $\zeta_C(2)^{-1} \zeta_C(10)^{-1}$.
	\end{prop}
	
	\section{Counting}

	\subsection{The Harder-Narasimhan polygon} \label{subsec:average_no_I-torsors}

We will first compute the average number of $I$-torsors, i.e. we want to estimate the following
	$$
		\lim_{d \to \infty} \frac{|\mathcal{M}_\mathcal{L}(k)|}{|\mathcal{A}_\mathcal{L}(k)|}.
	$$
	Let $d=\deg(\mathcal L)$. Since we are only interested in the behavior of this quotient when $d \to \infty$, when we do the computation below, we assume that $d \gg 0$. Note also that when $d \gg 0$, $|\mathcal{A}_\mathcal{L}(k)|$ is easy to compute using Riemann-Roch, since it is just the number of sections to $\mathcal{L}^{\otimes 4} \oplus \mathcal{L}^{\otimes 6}$. Indeed, we have
	$$
		|\mathcal{A}_\mathcal{L}(k)| = |H^0(C, \mathcal{L}^{\otimes 4}\oplus \mathcal{L}^{\otimes 6})| = q^{10 d +2(1-g)}, \qquad \text{when } d \gg 0.
	$$
	
We will count $|\mathcal{M}_\mathcal{L}(k)|$ by using the map 
$$\mathcal M_{\mathcal L} \to \Bun_G$$
and a partition of $\Bun_G(k)$ according to the Harder-Narasimhan polygon. 
Note that since $H^2(C, \Gm) = 0$ (see~\cite[p. 109]{milne_etale_1980}) every $G$-bundle $\mathcal E$ over $C$ can be lifted to a vector bundle $\mathcal F$ of rank $2$ which is well defined up to tensor twist by a line bundle.	If $\mathcal F$ is not semi-stable, then there is a unique tensor twist so that its Harder-Narasimhan filtration has the form
\[
	\xymatrix{
		0 \ar[r] & \mathcal{L}' \ar[r] & \mathcal F \ar[r] & \mathcal{O}_C \ar[r] & 0,
	}  \teq \label{eq:HN_F}
\]
with $\deg\mathcal{L}' > 0$. Note that after such normalization, $\mathcal F$ is determined uniquely by the associated $G$-bundle $\mathcal E$, and we will call $n=\deg\mathcal{L}'$ the unstable degree of $\mathcal E$.  It is not difficult to determine the size of the automorphism group of a $G$-bundle $\mathcal E$ of unstable degree $n$ large enough compared to the genus $g$
\[
		|\Aut_G(\mathcal E)| = (q-1)q^{n+1-g}. \teq \label{size-auto}
\]
	
	Let $\mathcal E$ be a $G$-torsor of unstable degree $n>0$; it can be lifted to a rank two vector bundle $\mathcal F$ fitting in the exact sequence \eqref{eq:HN_F}. We then have an associated 5-dimensional vector bundle $V({\mathcal E},\mathcal L)$:
	$$
		V({\mathcal E},\mathcal L) = (V\times^G \mathcal E) \otimes \mathcal{L}^{\otimes 2} = V(\mathcal{E}) \otimes \mathcal{L}^{\otimes 2} \cong \Sym^4 \mathcal F \otimes \det{}^{-2} \mathcal F \otimes \mathcal{L}^{\otimes 2}
	$$
	and $V({\mathcal E},\mathcal L)^\reg$ the regular part of $V({\mathcal E},\mathcal L)$. The filtration~\eqref{eq:HN_F} on $\mathcal F$ induces an obvious filtration on $V({\mathcal E},\mathcal L)$
	$$
		0 \subset \mathcal{F}_0 \subset \mathcal{F}_1 \subset \mathcal{F}_2 \subset \mathcal{F}_3 \subset \mathcal{F}_4 = V({\mathcal E},\mathcal L),
	$$
	where $\mathcal{F}_i/\mathcal{F}_{i-1} \cong \mathcal{L}'^{\otimes (2-i)} \otimes \mathcal{L}^{\otimes 2}$. 
	
	We will calculate the mass of the groupoid $\mathcal M_{\mathcal L}(k)$ in different ranges according to the integers $n$ and $d$:
	
		\paragraph{Case 1:}	$n>2d$. When $d$ is sufficiently large, the exact sequence~\eqref{eq:HN_F} splits, and we have $F \cong \mathcal{L}'\oplus \mathcal{O}_C$, which implies that
	\[ 
		V({\mathcal E},\mathcal L) \cong (\mathcal{L}'^{\otimes 2} \otimes\mathcal{L}^{\otimes 2}) \oplus (\mathcal{L}'\otimes \mathcal{L}^{\otimes 2}) \oplus \mathcal{L}^{\otimes 2} \oplus (\mathcal{L}'^{\otimes -1} \otimes \mathcal{L}^{\otimes 2}) \oplus (\mathcal{L}'^{\otimes -2} \oplus \mathcal{L}^{\otimes 2}). \teq \label{eq:V_direct_sum}
	\]
Because $n > 2d$, there is no non zero sections to of last 2 summands. Thus, any section $f$ to $V({\mathcal E},\mathcal L)$ will have the form
	$$
		f = c_0x^4 + c_1x^3y + c_2x^2y^2 = x^2(c_0x^2 + c_1xy + c_2y^2),
	$$
where $c_0,c_1,c_2$ are sections of the first three summands in \eqref{eq:V_direct_sum}. Observe that 
$$c_1^2-4c_0c_2 \in H^0(C, \mathcal{L}'^{\otimes 2}\otimes \mathcal{L}^{\otimes 4})$$ 
necessarily vanishes somewhere, and at that point, $f$ is of type $(2, 2)$, which is not in the regular part. Thus the subset of $\mathcal M_{\mathcal L}(k)$ with $n>2d$ is empty, and the contribution to the average is precisely 0.

	\paragraph{Case 2:} $n=2d$. If $\mathcal{L}'^{-1}\otimes \mathcal{L}^{\otimes 2}$ is not trivial, then since $\deg \mathcal{L}'^{-1}\otimes \mathcal{L}^{\otimes 2} = 0$, we have $H^0(C, \mathcal{L}'^{-1} \otimes \mathcal{L}^{\otimes 2}) = 0$. Thus, similar to the first case, there is no regular section. Hence, we need only to consider the case where $\mathcal{L}'\cong \mathcal{L}^{\otimes 2}$. In this case, when $d$ is sufficiently large, then $\mathcal F \cong \mathcal{L} \oplus \mathcal{O}$, and hence, $V({\mathcal E},\mathcal L) \cong \mathcal{L}^{\otimes 6} \oplus \mathcal{L}^{\otimes 4} \oplus \mathcal{L}^{\otimes 2} \oplus \mathcal{O}_C \oplus \mathcal{L}^{\otimes -2}$. Therefore, any section $f$ to $V({\mathcal E},\mathcal L)^\reg$ must have the form $(c_0, c_1, c_2, c_3, 0)$ with $c_3\neq 0$, or in a different notation
	$$
		f = c_0x^4 + c_1x^3y + c_2x^2y^2 +c_3xy^3,
	$$
since $H^0(C,\mathcal{L}^{\otimes -2})=0$. But now, we can bring this section to the form $y(x^3+axy^2 + by^3)$ via the automorphism
	$$
    \mtrix{1 & 0 \\ -c_2/3 & 1} \mtrix{c_3^{-1} & 0 \\ 0 & 1} \mtrix{0 & 1 \\ 1 & 0}, \qquad c_3 \neq 0.
	$$
We have thus shown that all regular sections in this case actually factor through the Weierstrass section. Thus, the contribution to the average of this case is precisely 1.
	
	\paragraph{Case 3:} $d<n<2d$. As above, where $d$ is sufficiently large, the exact sequence~\eqref{eq:HN_F} splits, and we have $\mathcal F \cong \mathcal{L}'\oplus \mathcal{O}$. This also splits $V({\mathcal E},\mathcal L)$ into a direct sum of $\mathcal{L}'^{\otimes (2-i)} \oplus \mathcal{L}^{\otimes 2}$ as in~\eqref{eq:V_direct_sum}. Using~\eqref{size-auto} and Riemann-Roch for the first three summands, we see that the mass of $\mathcal M_{\mathcal L}$ in this range is majorized by
	\begin{align*}
		&\quad \sum_{n=d+1}^{2d-1} \sum_{\deg\mathcal{L}' = n} \frac{q^{6d + 3n + 3(1-g)}|H^0(C, \mathcal{L}'^{-1}\otimes \mathcal{L}^{\otimes {2}})|}{(q-1) q^{n+1-g}} \\
		&= \sum_{n=d+1}^{2d-1} \frac{q^{6d+2n+2(1-g)} |\Sym_C^{2d-n}(\mathbb{F}_q)|}{q-1} \\
		&\leq \sum_{n=d+1}^{2d-1} \frac{Tq^{8d+n+2(1-g)}}{q-1} \tag{where $T$ is some constant} \\
		&= \frac{Tq^{8d+2(1-g)}}{q-1} \sum_{n=d+1}^{2d-1} q^n \\
		&\leq \frac{Tq^{10d+2(1-g)}}{q-1} \frac{1}{q-1}.
	\end{align*}
Thus, the contribution to the average is bounded above by
	$$
		\frac{Tq^{10d+2(1-g)}}{(q-1)^2 q^{10d+2(1-g)}} = \frac{T}{(q-1)^2}.
	$$
We also note that the implied constant $T$ only depends on the genus of $C$. 
	
	\paragraph{Case 4:} $d-g-1\leq n\leq d$. Similar to the above, when $d$ is sufficiently large, $\mathcal F \cong \mathcal{L}'\oplus \mathcal{O}_C$, which induces a splitting of the filtration on $V({\mathcal E},\mathcal L)$. We then see that
	$$
		\dim H^0(C, V({\mathcal E},\mathcal L)) = \sum_{i=0}^4 \dim H^0(C, \mathcal{L}'^{\otimes (2-1)} \otimes \mathcal{L}^{\otimes 2}) \leq 10d+ 5.
	$$
	Thus, if we let $A = |\Pic^0_{C/\mathbb{F}_q}(\mathbb{F}_q)| = |\Pic^i_{C/\mathbb{F}_q}(\mathbb{F}_q)|, \forall i$ (they are all equal since we assume that $C$ has an $\mathbb{F}_q$-rational point), then the mass of $\mathcal M_{\mathcal L}$ in this range is majorized by
	$$
		\sum_{n=d-g-1}^d \frac{A q^{10d+5}}{(q-1) q^{n+1-g}} = \frac{Aq^{10d+5}}{(q-1)q^{1-g}} \sum_{n=d-g-1}^d \frac{1}{q^n}.
	$$
The contribution to the average is therefore
	$$
		\frac{1}{q^{10d+2(1-g)}} \frac{A q^{10d+5}}{(q-1)q^{n+1-g}} \sum_{n=d-g-1}^d \frac{1}{q^n} = \frac{A q^{2+3g}}{q-1} \sum_{n=d-g-1}^{d} \frac{1}{q^n}.
	$$
But this goes to $0$ as $d$ goes to infinity, which means that there is no contribution to the average from this case.
	
	\paragraph{Case 5:} $0<n<d-g-1$ or $\mathcal F$ is semi-stable. By Riemann-Roch, we see that when $d$ is large enough,
	$$
		\dim H^0(C, V({\mathcal E},\mathcal L)) = \sum_{i=0}^4 \dim H^0(C, \mathcal{L}'^{\otimes (2-i)} \otimes \mathcal{L}^{\otimes 2}) = 10d+5(1-g).
	$$
Thus, when $0<n<d-g-1$ or $\mathcal F$ is semi-stable, we always have
$$
		|H^0(C, V({\mathcal E},\mathcal L))| = q^{10d+5(1-g)}.
$$
	
	To complete the computation in this case, we need one extra ingredient.
	\begin{prop} \label{prop:tamagawa_G_2}
		We have,
		$$
			|\Bun_G(\mathbb{F}_q)| = 2 q^{3(g-1)} \zeta_C(2).
		$$
	\end{prop}
	\begin{proof}
    This comes from the fact that the Tamagawa number of $G$ is 2.
  \end{proof}
	
	The contribution of this part to the average can now be computed as follows (here, the measure on $\Bun_G(\mathbb{F}_q)$ is just the counting measure, weighted by the sizes of the automorphism groups):
	\begin{align*}
		&\quad\lim_{d \to \infty} \disfrac{\int_{\Bun_G^{<d-g-1}(\mathbb{F}_q)} |H^0(C, V({\mathcal E},\mathcal L)^\reg)|d\mu}{|H^0(C, S\times^\Gm \mathcal{L})|} \\
		&= \lim_{d \to \infty} \disfrac{\int_{\Bun_G^{<d-g-1}(\mathbb{F}_q)} |H^0(C, V({\mathcal E},\mathcal L)^\reg)\,d\mu}{|H^0(C, \mathcal{L}^{\otimes 4})||H^0(C, \mathcal{L}^{\otimes 6})|} \\
		&= \lim_{d \to \infty} \disfrac{\int_{\Bun_G^{<d-g-1}(\mathbb{F}_q)} |H^0(C, V({\mathcal E},\mathcal L)^\reg)|\,d\mu}{\int_{\Bun_G^{<d-g-1}(\mathbb{F}_q)} |H^0(C, V({\mathcal E},\mathcal L))|\,d\mu} \disfrac{\int_{\Bun_G^{<d-g-1}(\mathbb{F}_q)} |H^0(C, V({\mathcal E},\mathcal L))|\,d\mu}{q^{10d+2(1-g)}} \\
		&= \lim_{d \to \infty} \disfrac{\int_{\Bun_G^{<d-g-1}(\mathbb{F}_q)} |H^0(C, V({\mathcal E},\mathcal L)^\reg)|\,d\mu}{|\Bun_G^{<d-g-1}(\mathbb{F}_q)| |H^0(C, V({\mathcal E},\mathcal L))|} \disfrac{|\Bun_G^{<d-g-1}(\mathbb{F}_q)| |H^0(C, V({\mathcal E},\mathcal L))|}{q^{10d+2(1-g)}} \\
		&= \lim_{d\to \infty} \disfrac{q^{10d+5(1-g)} \int_{\Bun_G^{<d-g-1}(\mathbb{F}_q)}\frac{|H^0(C, V({\mathcal E},\mathcal L)^\reg)|}{|H^0(C, V({\mathcal E},\mathcal L))|}\,d\mu}{q^{10d+2(1-g)}} \\
		&=\lim_{d\to \infty} q^{3(1-g)} \int_{\Bun_G^{<d-g-1}(\mathbb{F}_q)} \zeta_C(2)^{-1}\,d\mu \teq \label{eq:DCT} \\
		&= |\Bun_G(\mathbb{F}_q)| q^{3(1-g)} \zeta_C(2)^{-1} \\
		&= 2q^{3(g-1)}\zeta_C(2)q^{3(1-g)} \zeta_C(2)^{-1} \teq \label{eq:tamagawa_application} \\
		&= 2.
	\end{align*}
	The equality at~\eqref{eq:DCT} is due to the dominated convergent theorem, the fact that the integrand is bounded by 1, and the actual value of the limit given by proposition~\ref{prop:density_reg}. The equality at~\eqref{eq:tamagawa_application} is due to proposition~\ref{prop:tamagawa_G_2}.
	
	Altogether, we have
	$$
		\limsup_{d\to \infty} \frac{|\mathcal{M}_\mathcal{L}(k)|}{|\mathcal{A}_\mathcal{L}(k)|} \leq 3 + \frac{T}{(q-1)^2}.
	$$
	
	\subsection{The case $E[2](C)$ is non-trivial} \label{subsec:E[2](C)_nontrivial}
	We have estimated the average number of $I$-torsors. Proposition ~\ref{prop:inequalities} shows that we have a weaker link between the number of $I$-torsors and the size of the 2-Selmer groups when $E[2](C)$ is non-trivial. This subsection shows that the stronger inequality dominates our estimate of the average size of the 2-Selmer groups. In other words, we will show that the contribution from the case where $E[2](C)$ is non-trivial is 0.
	
	When $E[2](C)$ is non-trivial, where $E$ is given by $(\mathcal{L}, a, b)$, then we see that $x^3 + axz^2 + bz^3$ can be written in the form $(x+cz)(x^2-cxz+vz^2)$, where $c \in H^0(C, \mathcal{L}^{\otimes 2})$ and $v\in H^0(C, \mathcal{L}^{\otimes 4})$. In other words, $(a,b)$ is in the image of
	\begin{align*}
		H^0(C, \mathcal{L}^{\otimes 2}) \times H^0(C, \mathcal{L}^{\otimes 4}) &\to H^0(C, \mathcal{L}^{\otimes 4}) \times H^0(C, \mathcal{L}^{\otimes 6}) \\
		(c, v) &\mapsto (v-c^2, cv).
	\end{align*}
	
	When $d =\deg \mathcal{L}$ is sufficiently large, then we can use Riemann-Roch to compute the size of all the spaces involved and see that the number of all such pairs $(a,	b)$ is bounded by $q^{6d+2(1-g)}$.
	
	We know that the number of points on $C$, where the fiber of $E$ fails to be smooth is bounded by $\deg\Delta(a,b) = 10d$. Let $C'$ be the complement of these points in $C$, then from an argument similar to that of proposition~\ref{prop:inequalities}, we know that $|\Sel_2(E_{k(C)})| \leq |H^1(C', E[2])|$. Observe that we have the following map
	$$
		H^1(C', E[2]) \to \{\text{tame \etale{} covers of $C'$ of degree 4}\},
	$$
	where we know that the image lands in the tame part since the characteristic of our base field is at least $5$ and the cover is of degree 4.
	
	Note that the number of topological generators of $\pi_1^\tame(C')$ is bounded by $2g+10d$, since it is the profinite completion of the usual fundamental group of a lifting of $C'$ to $\mathbb{C}$. The right hand side is therefore bounded by $m4^{10d}$ where $m$ is some constant. Thus, to bound the size of $H^1(C', E[2])$, it suffices to bound the sizes of the fibers of this map.
	
	Suppose $T$ is a degree 4 \etale{} cover of $C'$, then giving $T$ the structure of an $E[2]$-torsor is the same as giving a map $E[2]\times_{C'} T \to T$ compatible with the structure maps to $C'$ satisfying certain properties. Since everything involved is proper and flat over $C'$, a map $E[2] \times_{C'} T \to T$ is determined uniquely by $(E[2]\times_{C'} T)_{k(C)} \to T_{k(C)}$, compatible with the structure maps to $\Spec k(C)$. Since everything here is \etale{} over $k(C)$, both sides they are in fact products of field extensions of $k(C)$. But now, we see at once that the number of such maps is bounded by the product of the dimension of both sides (as $k(C)$-vector spaces), which is $m' = 16\times 4$.
	
	The contribution of this case to the average is therefore bounded above by
	$$
		\frac{mm' q^{6d+2(1-g)} 4^{10d}}{q^{10d +2(1-g)}} = \frac{m'' 4^{10d}}{q^{4d}}.
	$$
	This goes to zero as $d$ goes to infinity if $q^4 > 4^{10}$ or equivalently, when $q>32$. This is the only source of restriction on the size of our base field.
	
	\subsection{The average in the transversal case}
	We will show that the average in this case is precisely 3, which is the content of theorem~\ref{thm:transversal}. The main observation is that we can completely ignore the range $d < n < 2d$.
	\begin{lem} \label{above}
		When $d<n<2d$, for all $s \in \Gamma(C, V({\mathcal E},\mathcal L))$, $\Delta(s) \in \Gamma(C, \mathcal{L}^{\otimes 12})$ is not square-free (i.e. not transversal).
	\end{lem}
	\begin{proof}
		As before, when $d$ is sufficiently large, $F$ splits, which induces a splitting of $V({\mathcal E},\mathcal L)$,
		$$
			V({\mathcal E},\mathcal L) \cong (\mathcal{L}^{\otimes 2}\otimes \mathcal{L}'^{\otimes 2}) \oplus (\mathcal{L}^{\otimes 2} \otimes \mathcal{L}') \oplus \mathcal{L}^{\otimes 2} \oplus (\mathcal{L}^{\otimes 2}\otimes \mathcal{L}'^{\otimes -1}) \oplus (\mathcal{L}^{\otimes 2} \otimes \mathcal{L}'^{\otimes -2}).
		$$
		And hence, we can write $s = (c_0, c_1, c_2, c_3, c_4)$ where each ``coordinate'' is a section of the line bundles in the summand above, in the same order. Clearly, $c_4=0$ since $\deg \mathcal{L}^{\otimes 2} \otimes \mathcal{L}'^{\otimes -2} < 0$. Moreover, since $\deg\mathcal{L}^{\otimes 2} \otimes \mathcal{L}'^{-1} > 0$, there exists a point $v\in |C|$ such that $c_3$ vanishes.
		
		But now, at $v$, the discriminant is
		$$
			\Delta 
			= -27c_0^2c_3^4 + 18c_0c_1c_2c_3^3-4c_0c_2^3c_3^2-4c_1^3c_3^3+c_1^2c_2^2c_3^2,
		$$
  which vanishes to order at least 2.
	\end{proof}
	
	The result then follows from the computation in subsection~\ref{subsec:average_no_I-torsors}. Indeed, we can ignore case 3 due to lemma \ref{above}, and use the density computation in propositions~\ref{prop:density_(a,b)_transversal} and~\ref{prop:density_V_reg_transversal} (instead of proposition~\ref{prop:density_reg}) in case 5. Note also that the Weierstrass curves we are counting over are automatically minimal, by the transversality condition.

	\subsection{The average size of 2-Selmer groups}
	We will now present the proof of theorem~\ref{thm:main}. We have,
	\begin{align*}
		&\quad \limsup_{\deg\mathcal{L}\to \infty} \disfrac{\sum_{\mathcal{L}(E) \cong \mathcal{L}} |\Sel_2(E_K)|}{|H^0(C, \mathcal{L}^{\otimes 4}\oplus \mathcal{L}^{\otimes 6})|} \\
		&= \limsup_{\deg\mathcal{L} \to \infty} \disfrac{\sum_{\substack{\mathcal{L}(E) \cong \mathcal{L} \\ E[2](C) = \{0\}}} |\Sel_2(E_K)| + \sum_{\substack{\mathcal{L}(E) \cong \mathcal{L} \\ E[2](C) \neq \{0\}}} |\Sel_2(E_K)|}{|\mathcal{A}_\mathcal{L}(k)|} \\
		&\leq \limsup_{\deg\mathcal{L}\to \infty} \disfrac{|\mathcal{M}_\mathcal{L}(k)| + \frac{3}{4}\sum_{\substack{\mathcal{L}(E) \cong \mathcal{L} \\ E[2](C) \neq \{0\}}} |\Sel_2(E_K)|}{|\mathcal{A}_\mathcal{L}(k)|} \tag{by proposition~\ref{prop:inequalities}}	\\
		&= \limsup_{\mathcal{L} \to \infty} \frac{|\mathcal{M}_\mathcal{L}(k)|}{|\mathcal{A}_\mathcal{L}(k)|} \tag{by subsection~\ref{subsec:E[2](C)_nontrivial}} \\
		&\leq 3 + \frac{T}{(q-1)^2}. \tag{by subsection~\ref{subsec:average_no_I-torsors}}
	\end{align*}
	
	Theorem~\ref{thm:main} then follows from this computation and the following remarks:
	\begin{enumerate}[\quad (i)]
			\setlength{\itemsep}{0pt}
  		\setlength{\parskip}{0pt}
	    \item We can exclude the Weierstrass curves $E$ such that $\Delta E = 0$, because~\cite[lemma 4.1]{poonen_squarefree_2003} shows that their contribution is 0.
	    \item To impose minimality condition of $E$ on the count, we use propositions~\ref{prop:density_minimal_A_B} and~\ref{prop:density_V_reg_minimal} in case 5, which still gives us the number 2. For case 3, the estimate picks up at most an extra factor of $\zeta_C(10)$. Other cases are not affected.
	    \item In the count, pairs of the form $(a, b)$ and $(c^4 a, c^6 b)$ with $c\in k^\times$ give the same isomorphism class. By rewriting, we get the expression in~\eqref{eq:AS_AR_formula}.
	\end{enumerate}
	
  For the lower bound, we have,
	\begin{align*}
		&\quad \liminf_{\deg\mathcal{L} \to \infty} \disfrac{\sum_{\mathcal{L}(E) \cong \mathcal{L}} |\Sel_2(E_K)|}{|H^0(C, \mathcal{L}^{\otimes 4} \oplus \mathcal{L}^{\otimes 6})|} \\
		&\geq\liminf_{\deg\mathcal{L}\to \infty} \disfrac{\sum_{\substack{E\text{ transversal}\\ \mathcal{L}(E) \cong \mathcal{L}}} |\Sel_2(E_K)|}{|H^0(C, \mathcal{L}^{\otimes 4} \oplus \mathcal{L}^{\otimes 6})|} \\
		&= 3\zeta_C(10)^{-1} \tag{from theorem~\ref{thm:transversal} and proposition~\ref{prop:density_(a,b)_transversal}.}
	\end{align*}
	The same remarks as above apply, and we conclude the proof of theorem~\ref{thm:main}.
	
	\bibliography{BCQ1014}
\end{document}